\documentclass[11pt]{amsart}
\usepackage[T1]{fontenc}
\usepackage[utf8]{inputenc}
\usepackage{lmodern}
\usepackage{amsmath,amssymb,amsthm,mathtools}
\usepackage{microtype}
\usepackage{enumitem}
\usepackage{aliascnt}
\usepackage{hyperref}
\usepackage[nameinlink,capitalise,noabbrev]{cleveref}
\usepackage{geometry}
\hypersetup{colorlinks=true,linkcolor=blue,citecolor=blue,urlcolor=blue}

\newtheorem{theorem}{Theorem}[section]
\newaliascnt{proposition}{theorem}
\newtheorem{proposition}[proposition]{Proposition}
\aliascntresetthe{proposition}
\newaliascnt{lemma}{theorem}
\newtheorem{lemma}[lemma]{Lemma}
\aliascntresetthe{lemma}
\newaliascnt{corollary}{theorem}
\newtheorem{corollary}[corollary]{Corollary}
\aliascntresetthe{corollary}
\newaliascnt{remark}{theorem}
\newtheorem{remark}[remark]{Remark}
\aliascntresetthe{remark}
\newaliascnt{problem}{theorem}
\newtheorem{problem}[problem]{Problem}
\aliascntresetthe{problem}
\theoremstyle{definition}
\newaliascnt{definition}{theorem}
\newtheorem{definition}[definition]{Definition}
\aliascntresetthe{definition}

\newcommand{\cI}{\mathcal I}
\newcommand{\cP}{\mathcal P}
\newcommand{\cA}{\mathcal A}
\newcommand{\cS}{\mathcal S}

\newcommand{\1}{\mathbf 1}

\newcommand{\lcm}{\operatorname{lcm}}

\newcommand{\mult}{\operatorname{mult}}

\title[Weighted density bounds]{Weighted density bounds for covering systems over Dedekind domains}
\author{Jihui Chen}
\address{School of Mathematics, South China University of Technology, Guangzhou 510640, China}
\email{202420129926@mail.scut.edu.cn}
\date{October 2026}


\begin{document}
\begin{abstract}
We study the density of uncovered sets arising from finite families of residue classes over Dedekind domains with finite residue norms.  Building on the distortion method of Balister, Bollob\'as, Morris, Sahasrabudhe and Tiba, we prove a general weighted density criterion in which the weight attached to a prime ideal is allowed to vary.  In the critical Mertens regime this yields, for every \(\varepsilon>0\), the admissible prime-ideal weight
\[
 1+\frac{(\log N\mathfrak p)^3(\log\log N\mathfrak p)^{1+\varepsilon}}{N\mathfrak p},
\]
up to a harmless modification at finitely many small prime ideals.  This is asymptotically smaller than the fixed weight \(1+(\log N\mathfrak p)^{3+\varepsilon}/N\mathfrak p\) appearing in the earlier density theorem.  The result is valid for systems of any fixed multiplicity and, for multiplicity one, requires the ideal moduli themselves to be distinct but not their norms.

We then place the analytic hypotheses in Knopfmacher's Axiom~A framework.  In the critical case of Axiom~A with exponent \(1\), the prime-ideal estimates needed by the weighted theorem are automatic.  More generally, under Axiom~A with exponent \(\rho>0\), we prove a weighted mean-value theorem on ideals; in the subcritical range \(0<\rho<1\), a first-moment argument shows that sufficiently large moduli cover an arbitrarily small proportion of the ambient space.  We also explain that the distortion construction itself only uses a compatible system of finite residue spaces with uniform fibres and a Chinese-remainder-type product structure, suggesting a possible extension beyond Dedekind domains.  Finally, we show that a direct passage from the second moment to higher integer moments does not lower the logarithmic exponent \(3\) within this moment hierarchy.
\end{abstract}

\maketitle

\section{Introduction}
A covering system is a finite family of residue classes whose union is the ambient arithmetic structure.  In the classical setting one considers arithmetic progressions
\[
 a_1+d_1\mathbb Z,\ldots,a_k+d_k\mathbb Z.
\]
The subject was initiated by Erd\H{o}s and has since led to a number of questions concerning the size, divisibility and distribution of the moduli.  We refer to \cite{BalisterSurvey2024} for a recent survey.

A density question of Erd\H{o}s and Graham asks for lower bounds on the density left uncovered by residue classes whose distinct moduli lie in a fixed multiplicative interval.  Filaseta, Ford, Konyagin, Pomerance and Yu proved a strong form of this statement in \cite{FFKPY2007}.  Hough's solution of the minimum modulus problem \cite{Hough2015} introduced a powerful probabilistic sieve, and Balister, Bollob\'as, Morris, Sahasrabudhe and Tiba subsequently developed a simpler distortion method which, among other applications, yields the following type of density bound: for sufficiently large distinct integer moduli, the uncovered density is bounded below in terms of
\[
 \sum_i \frac{\mu(d_i)}{d_i},
 \qquad
 \mu(p^r)=1+\frac{(\log p)^{3+\varepsilon}}{p}.
\]
See \cite[Theorem~1.1]{BBMST2022}.  The same paper points out that a bound depending only on such a weighted reciprocal sum fails when the prime-power weight is replaced by \(1+O(1/p)\), and explicitly raises the problem of characterising the weights for which a uniform density lower bound persists \cite[pp.~380--381]{BBMST2022}.

Covering systems have also been studied over algebraic number fields.  Kim extended density and reciprocal-sum results to rings of integers of number fields \cite{Kim2009}.  More recently Li, Wang and Yi developed the distortion method in the number-field setting \cite{LiWangYi2023}, while Li, Wang, Wang and Yi formulated the distortion method for Dedekind domains and applied it to global function fields \cite{LiWangWangYi2025}.  Thus the purpose of the present paper is not to claim the Dedekind-domain formulation of the distortion method itself as new.

Our main point is instead the choice of the distortion parameters.  We give a general sufficient condition on a variable prime-ideal weight.  This condition produces an explicit iterated-logarithmic improvement of the fixed weight in \cite{BBMST2022}, already when \(D=\mathbb Z\).  We formulate the result for a finite-norm Dedekind domain under the analytic prime-ideal estimates actually used in the proof.

Let \(D\) be a Dedekind domain such that \(D/I\) is finite for every nonzero ideal \(I\).  Write
\[
 N(I)=|D/I|
\]
for the norm, \(\cI(D)\) for the set of nonzero integral ideals and \(\cP(D)\) for the set of nonzero prime ideals.  We assume the Mertens-type upper bound
\begin{equation}\label{eq:MertensUpper}
 \sum_{N\mathfrak p\le x}\frac1{N\mathfrak p}
 \le \log\log x+C_D
\end{equation}
for all sufficiently large \(x\).  This holds, for example, for rings of integers of number fields by the Mertens theorem for number fields; see Rosen \cite{Rosen1999}.

For a positive function \(a:[2,\infty)\to(0,\infty)\), put
\begin{equation}\label{eq:AaDef}
 A_a(x)=\sum_{N\mathfrak p\le x}\frac{a(N\mathfrak p)}{N\mathfrak p},
\end{equation}
and define the multiplicative function \(\mu_a\) by
\begin{equation}\label{eq:muaDefIntro}
 \mu_a(\mathfrak p^r)=1+a(N\mathfrak p),\qquad r\ge1.
\end{equation}

The multiplicity of a finite family
\[
 \cA=\{a_\lambda+I_\lambda:\lambda\in\Lambda\}
\]
is
\[
 \mult(\cA)=\max_{I\in\cI(D)}\#\{\lambda:I_\lambda=I\}.
\]
Our principal result is the following.

\begin{theorem}[General weighted density criterion]\label{thm:mainWeighted}
Assume \eqref{eq:MertensUpper}.  Let \(a:[2,\infty)\to(0,\infty)\) satisfy \(a(x)\to0\), and suppose
\begin{equation}\label{eq:generalCriterion}
 \sum_{\mathfrak p\in\cP(D)}
 \frac{(\log N\mathfrak p)^3\exp\!\bigl(3A_a(N\mathfrak p)\bigr)}
 {(N\mathfrak p)^2a(N\mathfrak p)}<\infty.
\end{equation}
For every fixed integer \(s\ge1\) there exists \(M=M(D,a,s)\) such that the following holds.  If \(\cA=\{a_\lambda+I_\lambda\}_{\lambda\in\Lambda}\) is a finite family with \(\mult(\cA)\le s\) and \(N(I_\lambda)\ge M\) for every \(\lambda\), then its uncovered set
\[
 R=D\setminus\bigcup_{\lambda\in\Lambda}(a_\lambda+I_\lambda)
\]
has periodic density satisfying
\begin{equation}\label{eq:mainDensity}
 d(R)\ge \frac12\exp\left(-4\sum_{\lambda\in\Lambda}
 \frac{\mu_a(I_\lambda)}{N(I_\lambda)}\right).
\end{equation}
\end{theorem}

When \(s=1\), the ideal moduli are pairwise distinct.  Notice that this does \emph{not} require the numbers \(N(I_\lambda)\) to be distinct: different ideals of a Dedekind domain may have the same norm.

A convenient special case occurs when
\begin{equation}\label{eq:simpleA}
 \sum_{\mathfrak p}\frac{a(N\mathfrak p)}{N\mathfrak p}<\infty.
\end{equation}
Then \(A_a(x)=O(1)\), and the criterion in \cref{thm:mainWeighted} reduces to
\begin{equation}\label{eq:simpleB}
 \sum_{\mathfrak p}
 \frac{(\log N\mathfrak p)^3}
 {(N\mathfrak p)^2a(N\mathfrak p)}<\infty.
\end{equation}

For an explicit corollary, suppose also that
\begin{equation}\label{eq:primeIdealUpper}
 \pi_D(x):=\#\{\mathfrak p:N\mathfrak p\le x\}
 \ll_D \frac{x}{\log x}.
\end{equation}
Set \(\ell_2(x)=\max\{1,\log\log x\}\).

\begin{corollary}[An explicit iterated-logarithmic weight]\label{cor:explicitWeight}
Assume \eqref{eq:MertensUpper} and \eqref{eq:primeIdealUpper}.  For every \(\varepsilon>0\) and fixed \(s\ge1\), \cref{thm:mainWeighted} holds with
\begin{equation}\label{eq:explicitWeight}
 \mu_\varepsilon(\mathfrak p^r)
 =1+\frac{(\log N\mathfrak p)^3\ell_2(N\mathfrak p)^{1+\varepsilon}}
 {N\mathfrak p}.
\end{equation}
In particular, for \(D=\mathbb Z\) this prime-power weight is asymptotically smaller than
\(1+(\log p)^{3+\varepsilon}/p\).
\end{corollary}

The proof of \cref{cor:explicitWeight} is included below and uses only elementary exponential blocks together with \eqref{eq:primeIdealUpper}; no regularity assumption on an auxiliary slowly varying function is needed.

The hypotheses \eqref{eq:MertensUpper} and \eqref{eq:primeIdealUpper} are deliberately stated separately from any global counting axiom: they are the analytic inputs actually used by the weighted distortion theorem.  A natural structural condition implying them is Knopfmacher's Axiom~A for the ideal semigroup.  We develop this second layer in \cref{sec:axiomA}.  In particular, if
\begin{equation}\label{eq:linearIdealCountIntro}
 N_D(x):=\#\{\mathfrak a\in\cI(D):N\mathfrak a\le x\}
 =\kappa_Dx+O(x^\eta),\qquad 0\le\eta<1,
\end{equation}
then the ideal semigroup satisfies Axiom~A with exponent \(1\).  The abstract prime-number and Mertens theory then supplies the prime-ideal estimates used above.

In the same Axiom~A framework we prove a weighted mean-value theorem.  The linear case is the form needed for the interval application.

\begin{theorem}[Weighted ideal mean value, critical case]\label{thm:meanValueIntro}
Assume \eqref{eq:linearIdealCountIntro}.  Let \(a_{\mathfrak p}\ge0\) satisfy
\[
 \sum_{\mathfrak p}\frac{a_{\mathfrak p}}{N\mathfrak p}<\infty,
\]
and define \(\mu_a(\mathfrak p^r)=1+a_{\mathfrak p}\).  Then
\[
 \sum_{N\mathfrak a\le x}\mu_a(\mathfrak a)
 =\kappa_DH_ax+o(x),
 \qquad
 H_a=\prod_{\mathfrak p}\left(1+\frac{a_{\mathfrak p}}{N\mathfrak p}\right).
\]
Consequently, for every fixed \(K>1\),
\[
 \sum_{n<N\mathfrak a\le Kn}\frac{\mu_a(\mathfrak a)}{N\mathfrak a}
 =\kappa_DH_a\log K+o(1).
\]
\end{theorem}

In \cref{sec:axiomA} we prove the more general exponent-\(\rho\) form of this statement and show that \(0<\rho<1\) is subcritical for the covering problem: sufficiently large moduli then cover an arbitrarily small proportion of the ambient space.

The paper is organised as follows.  In \cref{sec:prelim} we set up periodic density, the ideal-theoretic notation, and the Axiom~A terminology.  In \cref{sec:distortion} we give the distortion framework, explicitly attributing the probabilistic mechanism to \cite{BBMST2022} and incorporating bounded multiplicity; we also record the structural features suggesting extensions beyond Dedekind domains.  In \cref{sec:weightedProof} we choose the distortion parameters and prove \cref{thm:mainWeighted}.  \Cref{sec:axiomA} derives consequences of Axiom~A, including the general weighted mean-value theorem, the interval application, and the subcritical case.  Finally, \cref{sec:higher} explains why direct higher integer moments do not improve the exponent \(3\) within the same moment hierarchy.

\section{Preliminaries}\label{sec:prelim}
For ideals \(I,J\subset D\), our divisibility convention is
\[
 I\mid J\quad\Longleftrightarrow\quad J\subseteq I.
\]
Thus
\[
 (I,J):=I+J
\]
is the greatest common divisor and
\[
 \lcm(I,J)=I\cap J.
\]
Every nonzero ideal has a unique factorisation into prime ideals, and the norm is multiplicative.

If \(S\subseteq D\) is periodic modulo a nonzero ideal \(Q\), define
\begin{equation}\label{eq:periodicDensity}
 d_Q(S)=\frac{|S\bmod Q|}{N(Q)}.
\end{equation}
If \(Q\mid Q'\) and \(S\) is already \(Q\)-periodic, then each class modulo \(Q\) has exactly \(N(Q')/N(Q)\) lifts modulo \(Q'\), so \(d_{Q'}(S)=d_Q(S)\).  We therefore write simply \(d(S)\).

\subsection{Axiom A and the abstract analytic input}\label{subsec:axiomAprelim}
If the set of integral ideals of bounded norm is finite, then the multiplicative monoid \(\cI(D)\), together with the norm \(N\), is a multiplicative arithmetical semigroup in the sense of Knopfmacher.  In particular, this finiteness is automatic under the counting hypothesis in the following definition.

\begin{definition}[Axiom A]\label{def:axiomA}
Let \(\rho>0\).  We say that the ideal semigroup of \(D\) satisfies \emph{Axiom A with exponent \(\rho\)} if there are constants \(\kappa_D>0\) and \(0\le\eta<\rho\) such that
\begin{equation}\label{eq:AxiomA}
 N_D(x):=\#\{\mathfrak a\in\cI(D):N\mathfrak a\le x\}
 =\kappa_Dx^\rho+O_D(x^\eta).
\end{equation}
\end{definition}

This is Knopfmacher's Axiom~A after the harmless normalisation of the norm described below; see \cite[Chs.~4 and~6]{Knopfmacher1975}.  In particular, the linear counting law \eqref{eq:linearIdealCountIntro} is Axiom~A with exponent \(\rho=1\).

\begin{proposition}[Abstract prime and Mertens input]\label{prop:AxiomAanalytic}
Assume that \(D\) satisfies Axiom~A with exponent \(\rho>0\).  Then
\begin{equation}\label{eq:abstractPNT}
 \pi_D(x)\sim \frac{x^\rho}{\rho\log x}
 \qquad (x\to\infty),
\end{equation}
and there is a constant \(C_{D,\rho}\) such that
\begin{equation}\label{eq:AxiomAMertensRho}
 \sum_{N\mathfrak p\le x}\frac1{(N\mathfrak p)^\rho}
 =\log\log x+C_{D,\rho}+O_D\!\left(\frac1{\log x}\right).
\end{equation}
In particular, if \(\rho=1\), then \eqref{eq:MertensUpper} and \eqref{eq:primeIdealUpper} hold.  If \(0<\rho<1\), then
\begin{equation}\label{eq:subcriticalPrimeSum}
 \sum_{\mathfrak p}\frac1{N\mathfrak p}<\infty.
\end{equation}
More precisely, in the subcritical range,
\begin{equation}\label{eq:subcriticalPrimeTail}
 \sum_{N\mathfrak p>X}\frac1{N\mathfrak p}
 \ll_{D,\rho}\frac{X^{\rho-1}}{\log X}.
\end{equation}
\end{proposition}

\begin{proof}
We first reduce the general exponent to the normalised exponent \(1\).  Define a new norm on the same free commutative monoid of ideals by
\[
 \|\mathfrak a\|_\rho:=N(\mathfrak a)^\rho.
\]
It is again multiplicative, and its counting function is
\[
 \widetilde N_D(y)
 =\#\{\mathfrak a:\|\mathfrak a\|_\rho\le y\}
 =N_D(y^{1/\rho})
 =\kappa_Dy+O_D\!\left(y^{\eta/\rho}\right),
\]
where \(\eta/\rho<1\).  Thus the renormalised arithmetical semigroup satisfies the usual linear form of Axiom~A.  Knopfmacher's abstract prime number theorem and generalized Mertens formula for this normalised setting give
\[
 \widetilde\pi_D(y)\sim\frac{y}{\log y}
\]
and
\[
 \sum_{\|\mathfrak p\|_\rho\le y}\frac1{\|\mathfrak p\|_\rho}
 =\log\log y+\widetilde C_D+O_D\!\left(\frac1{\log y}\right);
\]
see \cite[Ch.~6]{Knopfmacher1975} for the abstract prime number theorem and \cite{Knopfmacher1975} for the generalized Mertens formula under Axiom~A.

Taking \(y=x^\rho\) gives
\[
 \widetilde\pi_D(x^\rho)=\pi_D(x)
 \sim \frac{x^\rho}{\log(x^\rho)}
 =\frac{x^\rho}{\rho\log x},
\]
which is \eqref{eq:abstractPNT}.  Likewise,
\[
 \sum_{N\mathfrak p\le x}\frac1{(N\mathfrak p)^\rho}
 =\log\log(x^\rho)+\widetilde C_D
  +O_D\!\left(\frac1{\log x}\right).
\]
Since
\[
 \log\log(x^\rho)=\log\log x+\log\rho
\]
for all sufficiently large \(x\), the constant \(\log\rho\) may be absorbed into \(\widetilde C_D\), proving \eqref{eq:AxiomAMertensRho}.

Suppose now that \(\rho=1\).  Then \eqref{eq:AxiomAMertensRho} gives
\[
 \sum_{N\mathfrak p\le x}\frac1{N\mathfrak p}
 =\log\log x+C_{D,1}+O_D\!\left(\frac1{\log x}\right),
\]
so, after increasing the constant to absorb the finitely many small values of \(x\), we obtain \eqref{eq:MertensUpper}.  Also \eqref{eq:abstractPNT} implies
\[
 \pi_D(x)\ll_D\frac{x}{\log x},
\]
which is \eqref{eq:primeIdealUpper}.

Finally assume \(0<\rho<1\).  By \eqref{eq:abstractPNT}, there exist \(X_0\ge2\) and \(C_{D,\rho}>0\) such that
\begin{equation}\label{eq:subcriticalPiUpper}
 \pi_D(t)\le C_{D,\rho}\frac{t^\rho}{\log t}
 \qquad (t\ge X_0).
\end{equation}
For \(Y>X\ge X_0\), partial summation with \(t^{-1}\) gives
\[
 \sum_{X<N\mathfrak p\le Y}\frac1{N\mathfrak p}
 =\frac{\pi_D(Y)}Y-\frac{\pi_D(X)}X
   +\int_X^Y\frac{\pi_D(t)}{t^2}\,dt.
\]
Since \(\rho<1\),
\[
 \frac{\pi_D(Y)}Y
 \ll_{D,\rho}\frac{Y^{\rho-1}}{\log Y}\to0.
\]
Letting \(Y\to\infty\) and discarding the non-positive term \(-\pi_D(X)/X\),
\[
 \sum_{N\mathfrak p>X}\frac1{N\mathfrak p}
 \le \int_X^\infty\frac{\pi_D(t)}{t^2}\,dt
 \ll_{D,\rho}\int_X^\infty\frac{t^{\rho-2}}{\log t}\,dt.
\]
For \(t\ge X\), \((\log t)^{-1}\le(\log X)^{-1}\), hence
\[
 \int_X^\infty\frac{t^{\rho-2}}{\log t}\,dt
 \le \frac1{\log X}\int_X^\infty t^{\rho-2}\,dt
 =\frac{X^{\rho-1}}{(1-\rho)\log X}.
\]
This proves \eqref{eq:subcriticalPrimeTail}; letting \(X\to\infty\) proves \eqref{eq:subcriticalPrimeSum}.
\end{proof}

\begin{remark}\label{rem:numberFieldAxiomA}
For the ring of integers of a number field, the classical ideal-counting theorem gives Axiom~A with exponent \(1\).  Thus in that setting the analytic hypotheses used by \cref{thm:mainWeighted,cor:explicitWeight} are not additional assumptions but consequences of the standard ideal-counting theory.  More generally, Axiom~A provides a convenient structural sufficient condition for the analytic input while the weighted theorem itself remains stated under the weaker hypotheses actually used in its proof.
\end{remark}

Let \(\cA=\{a_\lambda+I_\lambda\}_{\lambda\in\Lambda}\) be a finite family and put
\[
 Q=\lcm_{\lambda\in\Lambda}I_\lambda.
\]
Write
\begin{equation}\label{eq:Qfactor}
 Q=\prod_{i=1}^{n}\mathfrak p_i^{v_i},
\end{equation}
where the prime ideals are ordered by nondecreasing norm; ties are broken arbitrarily.  Put
\[
 Q_0=D,\qquad Q_i=\prod_{j=1}^{i}\mathfrak p_j^{v_j}.
\]
Let
\[
 \Lambda_i=\{\lambda:I_\lambda\mid Q_i,\ I_\lambda\nmid Q_{i-1}\},
 \qquad
 B_i=\bigcup_{\lambda\in\Lambda_i}(a_\lambda+I_\lambda).
\]

\section{The distortion framework}\label{sec:distortion}
The construction in this section is the ideal-theoretic version of the distortion method of Balister--Bollob\'as--Morris--Sahasrabudhe--Tiba \cite{BBMST2022}.  The point of recording it is to make the later choice of variable weights precise and to track multiplicity.

For \(1\le i\le n\), let
\[
 \pi_i:D/Q_i\to D/Q_{i-1}
\]
be the natural projection and, for \(x\in D/Q_{i-1}\), let
\[
 F_i(x)=\pi_i^{-1}(x).
\]
Since \(Q_{i-1}\) and \(\mathfrak p_i^{v_i}\) are coprime, the Chinese remainder theorem gives
\[
 D/Q_i\cong D/Q_{i-1}\times D/\mathfrak p_i^{v_i},
\]
so
\begin{equation}\label{eq:fibresize}
 |F_i(x)|=N(\mathfrak p_i)^{v_i}.
\end{equation}

Let \(P_0\) be the probability measure on the one-point quotient \(D/Q_0\); after uniform extension to \(D/Q\), this is the uniform measure.  Suppose \(P_{i-1}\) has been defined on \(D/Q_{i-1}\).  Extend it uniformly to each fibre of \(D/Q_i\), writing
\[
 \widetilde P_{i-1}(y)=\frac{P_{i-1}(x)}{|F_i(x)|},\qquad y\in F_i(x).
\]
Define
\begin{equation}\label{eq:alphaDef}
 \alpha_i(x)=\frac{|F_i(x)\cap B_i|}{|F_i(x)|}.
\end{equation}
Fix parameters \(\theta_i\in[0,1/2]\).  If \(\alpha_i(x)\le\theta_i\), set
\begin{equation}\label{eq:PiCase1}
 P_i(y)=
 \begin{cases}
 0,& y\in B_i,\\[1mm]
 \widetilde P_{i-1}(y)/(1-\alpha_i(x)),&y\notin B_i.
 \end{cases}
\end{equation}
If \(\alpha_i(x)>\theta_i\), set
\begin{equation}\label{eq:PiCase2}
 P_i(y)=
 \begin{cases}
 \dfrac{\alpha_i(x)-\theta_i}{\alpha_i(x)(1-\theta_i)}\widetilde P_{i-1}(y),&y\in B_i,\\[3mm]
 \dfrac1{1-\theta_i}\widetilde P_{i-1}(y),&y\notin B_i.
 \end{cases}
\end{equation}
A direct computation on each fibre shows that its total mass is preserved.

\begin{lemma}\label{lem:measureBasic}
If \(S\) is \(Q_{i-1}\)-measurable, then \(P_i(S)=P_{i-1}(S)\).  For arbitrary \(S\subseteq D/Q\),
\[
 P_i(S)\le\frac1{1-\theta_i}P_{i-1}(S),
\]
and if \(S\subseteq B_i\), then \(P_i(S)\le P_{i-1}(S)\).
\end{lemma}
\begin{proof}
The first assertion follows by summing the fibre identity
\[
 \sum_{y\in F_i(x)}P_i(y)=P_{i-1}(x).
\]
The other two assertions follow pointwise from the multipliers in \eqref{eq:PiCase1}--\eqref{eq:PiCase2}.
\end{proof}

For a nonnegative function \(f\) on \(D/Q\), write \(E_i[f]\) for expectation with respect to the uniformly extended \(P_i\).  Set
\[
 M_i^{(1)}=E_{i-1}[\alpha_i],\qquad M_i^{(2)}=E_{i-1}[\alpha_i^2].
\]
Define the multiplicative function \(\nu\) on divisors of \(Q\) by
\begin{equation}\label{eq:nuDef}
 \nu(\mathfrak p_j^r)=\frac1{1-\theta_j},\qquad r\ge1.
\end{equation}
For each \(i\), put
\begin{equation}\label{eq:betaDef}
 \beta_i=
 \begin{cases}
 M_i^{(1)},&\theta_i=0,\\[1mm]
 \min\left\{M_i^{(1)},\dfrac{M_i^{(2)}}{4\theta_i(1-\theta_i)}\right\},&\theta_i>0,
 \end{cases}
 \qquad
 \eta_*=\sum_i\beta_i.
\end{equation}
This convention explicitly removes the zero-denominator issue at \(\theta_i=0\).

\begin{lemma}\label{lem:BiMoment}
For every \(i\), \(P_i(B_i)\le M_i^{(1)}\).  If \(\theta_i>0\), then
\[
 P_i(B_i)\le\frac{M_i^{(2)}}{4\theta_i(1-\theta_i)}.
\]
\end{lemma}
\begin{proof}
The first inequality follows from \cref{lem:measureBasic} and the identity \(P_{i-1}(B_i)=E_{i-1}[\alpha_i]\).  For the second, \eqref{eq:PiCase1}--\eqref{eq:PiCase2} give
\[
 P_i(B_i)=\frac1{1-\theta_i}E_{i-1}[(\alpha_i-\theta_i)_+].
\]
Use \((u-d)_+\le u^2/(4d)\) for \(u,d>0\).
\end{proof}

\begin{lemma}[Residue-class measure bound]\label{lem:classBound}
For \(0\le i\le n\), every \(I\mid Q\), and every class \(a+I\),
\begin{equation}\label{eq:classBound}
 P_i(a+I)\le\frac{\nu((I,Q_i))}{N(I)}.
\end{equation}
\end{lemma}
\begin{proof}
Write
\[
 I=I_{\le i}I_{>i},\qquad
 I_{\le i}=\prod_{j\le i}\mathfrak p_j^{v_{\mathfrak p_j}(I)},
 \quad
 I_{>i}=\prod_{j>i}\mathfrak p_j^{v_{\mathfrak p_j}(I)}.
\]
Then \(I_{\le i}=(I,Q_i)\), and the two factors are coprime.  Since \(P_i\) is extended uniformly to the prime coordinates with index greater than \(i\), the Chinese remainder theorem gives
\begin{equation}\label{eq:futureUniform}
 P_i(a+I)=\frac{P_i(a+I_{\le i})}{N(I_{>i})}.
\end{equation}
It is therefore enough to prove
\[
 P_i(a+I_{\le i})\le\frac{\nu(I_{\le i})}{N(I_{\le i})}.
\]
We induct on \(i\).  The case \(i=0\) is immediate.  If \(\mathfrak p_i\nmid I\), then \(I_{\le i}=I_{\le i-1}\), the class is \(Q_{i-1}\)-measurable, and \cref{lem:measureBasic} plus induction applies.  Otherwise write \(I_{\le i}=J\mathfrak p_i^e\) with \(J\mid Q_{i-1}\), \(e\ge1\).  By the pointwise distortion bound and the uniform extension of \(P_{i-1}\) in the \(\mathfrak p_i\)-coordinate,
\[
 P_i(a+J\mathfrak p_i^e)
 \le\frac1{1-\theta_i}\frac{P_{i-1}(a+J)}{N(\mathfrak p_i)^e}
 \le\frac{\nu(J)}{(1-\theta_i)N(J)N(\mathfrak p_i)^e}.
\]
This is the desired estimate, and \eqref{eq:futureUniform} completes the proof.
\end{proof}

Define the pointwise distortion on \(D/Q\) by
\[
 \Delta_i(x)=\max\left\{0,\log\frac{P_i(x)}{P_0(x)}\right\}.
\]

\begin{lemma}[Average distortion]\label{lem:averageDistortion}
For every \(i\),
\begin{equation}\label{eq:averageDistortion}
 E_i[\Delta_i]
 \le2\sum_{\substack{\lambda\in\Lambda\\ I_\lambda\mid Q_i}}
 \frac{\nu(I_\lambda)}{N(I_\lambda)}.
\end{equation}
\end{lemma}
\begin{proof}
For a point of positive \(P_j\)-mass, let \(g_j\) denote the multiplicative factor by which the uniformly extended measure \(P_{j-1}\) is changed at step \(j\); on atoms of zero \(P_j\)-mass we may set \(g_j=1\), since they make no contribution to \(E_i[\Delta_i]\) for any \(i\ge j\).  From \eqref{eq:PiCase1}--\eqref{eq:PiCase2}, whenever \(g_j>1\),
\[
 \log g_j\le2\alpha_j.
\]
Indeed, if \(\alpha_j\le1/2\), then \(g_j\le(1-\alpha_j)^{-1}\) and \(-\log(1-\alpha_j)\le2\alpha_j\); if \(\alpha_j>1/2\), then \(g_j\le(1-\theta_j)^{-1}\le2\), so \(\log g_j\le\log2\le2\alpha_j\).  On the support of \(P_i\), the density ratio telescopes as \(P_i/P_0=\prod_{j\le i}g_j\), and therefore
\[
 \Delta_i\le\sum_{j\le i}(\log g_j)_+\le2\sum_{j\le i}\alpha_j.
\]
Taking \(P_i\)-expectations and using the repeated measure-preservation property for the \(Q_{j-1}\)-measurable function \(\alpha_j\),
\[
 E_i[\Delta_i]\le2\sum_{j\le i}E_i[\alpha_j]
 =2\sum_{j\le i}E_{j-1}[\alpha_j]
 =2\sum_{j\le i}P_{j-1}(B_j).
\]
For \(\lambda\in\Lambda_j\), write \(I_\lambda=H_\lambda\mathfrak p_j^{r_\lambda}\) with \(H_\lambda\mid Q_{j-1}\).  By the union bound and \cref{lem:classBound},
\[
 P_{j-1}(B_j)
 \le\sum_{\lambda\in\Lambda_j}\frac{\nu(H_\lambda)}{N(I_\lambda)}
 \le\sum_{\lambda\in\Lambda_j}\frac{\nu(I_\lambda)}{N(I_\lambda)}.
\]
Summing over \(j\) proves the claim.
\end{proof}

\begin{theorem}[Abstract distortion criterion]\label{thm:abstractDistortion}
If \(\eta_*<1\), then \(\cA\) does not cover \(D\), and
\begin{equation}\label{eq:abstractDensity}
 d(R)\ge(1-\eta_*)\exp\left(
 -\frac2{1-\eta_*}\sum_{\lambda\in\Lambda}
 \frac{\nu(I_\lambda)}{N(I_\lambda)}\right).
\end{equation}
In particular, if \(\eta_*\le1/2\), then
\[
 d(R)\ge\frac12\exp\left(-4\sum_\lambda
 \frac{\nu(I_\lambda)}{N(I_\lambda)}\right).
\]
\end{theorem}
\begin{proof}
Since \(B_i\) is \(Q_i\)-measurable, \cref{lem:measureBasic} gives \(P_n(B_i)=P_i(B_i)\).  Thus
\[
 1-P_n(R)\le\sum_iP_i(B_i)\le\eta_*,
\]
by \cref{lem:BiMoment}.  Therefore \(P_n(R)\ge1-\eta_*\).  Also, by definition of \(\Delta_n\),
\[
 P_0(R)\ge E_n[\1_R e^{-\Delta_n}].
\]
Conditioning on \(R\) and applying Jensen's inequality gives
\[
 E_n[\1_R e^{-\Delta_n}]
 \ge P_n(R)\exp\left(-\frac{E_n[\Delta_n]}{P_n(R)}\right).
\]
Now use \cref{lem:averageDistortion}.  Finally \(P_0(R)=d(R)\) because \(P_0\) is uniform on \(D/Q\).
\end{proof}

We next record the moment bounds.  Put \(q_i=N\mathfrak p_i\).  Every modulus exposed at stage \(i\) has the form
\[
 I_\lambda=H_\lambda\mathfrak p_i^{r_\lambda},
 \qquad H_\lambda\mid Q_{i-1},\quad 1\le r_\lambda\le v_i.
\]
For \(x\in D/Q_{i-1}\), the fibre \(F_i(x)\) has cardinality
\(q_i^{v_i}\).  To compute the contribution of a single new residue
class, use the coprime factorisation
\(I_\lambda=H_\lambda\mathfrak p_i^{r_\lambda}\).  By the Chinese
remainder theorem, the congruence
\(y\equiv a_\lambda\pmod{I_\lambda}\), with
\(y\in F_i(x)\), has no solution unless
\(x\equiv a_\lambda\pmod{H_\lambda}\).  If this latter congruence
holds, precisely \(q_i^{v_i-r_\lambda}\) of the \(q_i^{v_i}\)
points in the fibre satisfy the additional congruence
\(y\equiv a_\lambda\pmod{\mathfrak p_i^{r_\lambda}}\).  Hence
\[
 \frac{|F_i(x)\cap(a_\lambda+I_\lambda)|}{|F_i(x)|}
 =q_i^{-r_\lambda}\1_{\{x\equiv a_\lambda\pmod{H_\lambda}\}}.
\]
The sets \(F_i(x)\cap(a_\lambda+I_\lambda)\) need not be disjoint
as \(\lambda\) varies.  Since
\(B_i=\bigcup_{\lambda\in\Lambda_i}(a_\lambda+I_\lambda)\),
we apply the union bound to their cardinalities to obtain
\begin{equation}\label{eq:alphaUnion}
 \alpha_i(x)=\frac{|F_i(x)\cap B_i|}{|F_i(x)|}
 \le\sum_{\lambda\in\Lambda_i}
 q_i^{-r_\lambda}\1_{\{x\equiv a_\lambda\pmod{H_\lambda}\}}.
\end{equation}
This is the ideal-theoretic counterpart of the fibre-counting step
in the proof of \cite[Lemma~3.6]{BBMST2022}; the argument is recorded
here because it is the starting point for all of the moment bounds below.

\begin{lemma}[General moment bound]\label{lem:generalMoment}
If \(\mult(\cA)\le s\), then for every integer \(k\ge1\),
\begin{equation}\label{eq:generalMoment}
 E_{i-1}[\alpha_i^k]
 \le\frac{s^k}{(q_i-1)^k}
 \sum_{H_1,\ldots,H_k\mid Q_{i-1}}
 \frac{\nu(\lcm(H_1,\ldots,H_k))}
 {N(\lcm(H_1,\ldots,H_k))}.
\end{equation}
Moreover,
\begin{equation}\label{eq:sharpFirstMoment}
 M_i^{(1)}\le
 \sum_{\lambda\in\Lambda_i}q_i^{-r_\lambda}
 \frac{\nu(H_\lambda)}{N(H_\lambda)}.
\end{equation}
\end{lemma}
\begin{proof}
Expand the \(k\)-th power of \eqref{eq:alphaUnion}.  The simultaneous congruences modulo \(H_1,\ldots,H_k\) are either incompatible or define one residue class modulo \(L=\lcm(H_1,\ldots,H_k)\).  In the latter case \cref{lem:classBound} bounds its \(P_{i-1}\)-measure by \(\nu(L)/N(L)\).  Each underlying ideal modulus supports at most \(s\) residue classes, so a \(k\)-tuple of underlying moduli gives at most \(s^k\) residue-class tuples.  Finally enlarge the exponent sums and use \(\sum_{r\ge1}q_i^{-r}=1/(q_i-1)\).  The case \(k=1\) without enlargement gives \eqref{eq:sharpFirstMoment}.
\end{proof}

\begin{proposition}[First and second moments]\label{prop:firstSecond}
For every \(i\),
\begin{align}
 M_i^{(1)}&\le\frac{s}{q_i-1}
 \prod_{j<i}\left(1+\frac1{(1-\theta_j)(q_j-1)}\right),\label{eq:firstMomentProduct}\\
 M_i^{(2)}&\le\frac{s^2}{(q_i-1)^2}
 \prod_{j<i}\left(1+\frac{3q_j-1}{(1-\theta_j)(q_j-1)^2}\right).
 \label{eq:secondMomentProduct}
\end{align}
\end{proposition}
\begin{proof}
For \(k=1\), factor the divisor sum in \cref{lem:generalMoment} prime by prime and use
\[
 \sum_{r=1}^{v_j}q_j^{-r}\le\frac1{q_j-1}.
\]
For \(k=2\), if the maximum of the two local exponents at \(\mathfrak p_j\) equals \(r\ge1\), there are \((r+1)^2-r^2=2r+1\) ordered pairs.  Hence the local factor is bounded by
\[
 1+\frac1{1-\theta_j}\sum_{r\ge1}\frac{2r+1}{q_j^r}
 =1+\frac{3q_j-1}{(1-\theta_j)(q_j-1)^2}.
\]
\end{proof}

\begin{remark}[A possible scope beyond Dedekind domains]\label{rem:beyondDedekind}
The probabilistic construction above uses less structure than the full ring structure of a Dedekind domain.  Its essential inputs are: a multiplicative system of moduli with unique prime factorisation; finite residue spaces \(X_I\) of cardinality \(N(I)\); compatible surjective maps \(X_J\to X_I\) whenever \(I\mid J\), with uniform fibres of size \(N(J)/N(I)\); and a Chinese-remainder-type product decomposition for coprime moduli.  Under such axioms the fibre-by-fibre definition of the measures \(P_i\), the moment estimates, and the distortion argument have formal analogues.  We do not develop an abstract ``residue arithmetical semigroup'' theory here, since doing so would require additional definitions unrelated to the weighted application.  The observation nevertheless suggests that the distortion method should extend beyond the Dedekind-domain setting whenever a compatible finite residue theory is available.
\end{remark}

\section{Proof of the weighted density theorem}\label{sec:weightedProof}
We now choose the distortion parameters.  The following estimate is the only place where the exponent \(3\) appears.

\begin{lemma}\label{lem:eulerProduct}
Suppose \(\theta_j=0\), or
\[
 \frac1{1-\theta_j}=1+a(q_j),\qquad a(q_j)\le1.
\]
Then, uniformly in the finite system,
\begin{equation}\label{eq:eulerProductBound}
 \prod_{j<i}\left(1+\frac{3q_j-1}{(1-\theta_j)(q_j-1)^2}\right)
 \ll_D (\log q_i)^3e^{3A_a(q_i)}.
\end{equation}
\end{lemma}
\begin{proof}
Write
\[
 b(q)=\frac{3q-1}{(q-1)^2},\qquad
 u_j=\frac{3q_j-1}{(1-\theta_j)(q_j-1)^2}
     =\frac{b(q_j)}{1-\theta_j}.
\]
The exact identity
\[
 b(q)-\frac3q=\frac{5q-3}{q(q-1)^2}
\]
implies, for every \(q\ge2\),
\[
 0\le b(q)-\frac3q\le\frac{20}{q^2}.
\]
If \(\theta_j=0\), this gives
\(u_j\le 3/q_j+20/q_j^2\).  Otherwise, by hypothesis,
\((1-\theta_j)^{-1}=1+a(q_j)\) and \(0<a(q_j)\le1\); hence
\[
 \begin{aligned}
 u_j
  &=(1+a(q_j))\left(\frac3{q_j}
       +\frac{5q_j-3}{q_j(q_j-1)^2}\right)\\
  &\le \frac3{q_j}+\frac{3a(q_j)}{q_j}
       +\frac{40}{q_j^2}.
 \end{aligned}
\]
Thus in both cases we have the uniform bound
\begin{equation}\label{eq:eulerTermEstimate}
 0\le u_j\le \frac3{q_j}+\frac{3a(q_j)}{q_j}
                          +\frac{40}{q_j^2}.
\end{equation}
We next verify that the last error term is summable over all prime
ideals, using only the Mertens-type hypothesis \eqref{eq:MertensUpper}.
After increasing its constant to cover finitely many small norms,
write
\[
 S_D(X)=\sum_{N\mathfrak p\le X}\frac1{N\mathfrak p}
      \le\log\log X+C_D\qquad(X\ge2).
\]
For every sufficiently large integer \(m\),
\[
 \begin{aligned}
 \sum_{2^m<N\mathfrak p\le2^{m+1}}\frac1{(N\mathfrak p)^2}
 &\le2^{-m}\sum_{N\mathfrak p\le2^{m+1}}
                      \frac1{N\mathfrak p}\\
 &\le2^{-m}\bigl(\log\log(2^{m+1})+C_D\bigr).
 \end{aligned}
\]
The last expression is \(O_D(2^{-m}\log(m+2))\), which is
summable in \(m\).  The bounded initial range contains only finitely
many prime ideals by the same hypothesis.  Consequently,
\begin{equation}\label{eq:primeSquareSum}
 C_{2,D}:=\sum_{\mathfrak p}\frac1{(N\mathfrak p)^2}<\infty.
\end{equation}
Now put
\[
 \mathcal P_i=\prod_{j<i}(1+u_j).
\]
Because \(u_j\ge0\) and \(\log(1+u)\le u\) for \(u\ge0\),
\eqref{eq:eulerTermEstimate} yields
\[
 \begin{aligned}
 \log\mathcal P_i
 &=\sum_{j<i}\log(1+u_j)\\
 &\le3\sum_{j<i}\frac1{q_j}
       +3\sum_{j<i}\frac{a(q_j)}{q_j}
       +40\sum_{j<i}\frac1{q_j^2}.
 \end{aligned}
\]
The primes are ordered by nondecreasing norm, with ties allowed.
Therefore \(j<i\) implies \(q_j\le q_i\), and hence
\[
 \sum_{j<i}\frac1{q_j}\le S_D(q_i)
 \le\log\log q_i+C_D,
 \qquad
 \sum_{j<i}\frac{a(q_j)}{q_j}\le A_a(q_i).
\]
The remaining term is at most \(40C_{2,D}\).  Combining these
estimates gives
\[
 \log\mathcal P_i
 \le3\log\log q_i+3A_a(q_i)+3C_D+40C_{2,D}.
\]
Exponentiating, we conclude that
\[
 \mathcal P_i\le
 \exp(3C_D+40C_{2,D})\,(\log q_i)^3e^{3A_a(q_i)},
\]
which is \eqref{eq:eulerProductBound}.  The constant depends only on
the prime-norm distribution of \(D\) (and the chosen Mertens constant),
not on \(i\), the finite residue-class family, its multiplicity, or the
particular admissible function \(a\).
\end{proof}

\begin{proof}[Proof of \cref{thm:mainWeighted}]
Fix \(s\).  By \eqref{eq:generalCriterion} and \(a(x)\to0\), choose \(Y\) so large that \(a(x)\le1\) for \(x>Y\) and
\begin{equation}\label{eq:chooseY}
 C_Ds^2\sum_{N\mathfrak p>Y}
 \frac{(\log N\mathfrak p)^3e^{3A_a(N\mathfrak p)}}{(N\mathfrak p)^2a(N\mathfrak p)}
 \le\frac14,
\end{equation}
where \(C_D\) is large enough for \cref{prop:firstSecond,lem:eulerProduct}.
Set
\[
 \theta_i=
 \begin{cases}
 0,&q_i\le Y,\\[1mm]
 \dfrac{a(q_i)}{1+a(q_i)},&q_i>Y.
 \end{cases}
\]
For \(q_i>Y\),
\[
 \frac1{1-\theta_i}=1+a(q_i),\qquad
 \theta_i(1-\theta_i)=\frac{a(q_i)}{(1+a(q_i))^2},
\]
so
\[
 \frac1{4\theta_i(1-\theta_i)}\le\frac1{a(q_i)}.
\]
By \cref{prop:firstSecond,lem:eulerProduct},
\[
 \frac{M_i^{(2)}}{4\theta_i(1-\theta_i)}
 \ll_D s^2\frac{(\log q_i)^3e^{3A_a(q_i)}}{q_i^2a(q_i)}.
\]
Thus \eqref{eq:chooseY} gives
\begin{equation}\label{eq:largePrimeQuarter}
 \sum_{q_i>Y}\frac{M_i^{(2)}}{4\theta_i(1-\theta_i)}\le\frac14.
\end{equation}

For the small-prime stages, let \(\cS_Y\) be the set of ideals all of whose prime factors have norm at most \(Y\).  Since only finitely many such prime ideals occur,
\[
 \sum_{I\in\cS_Y}\frac1{N(I)}
 =\prod_{N\mathfrak p\le Y}\left(1-\frac1{N\mathfrak p}\right)^{-1}<\infty.
\]
If \(q_i\le Y\), then every \(j\le i\) also satisfies \(q_j\le Y\), so all relevant \(\theta_j\) vanish and \(\nu(H)=1\) for \(H\mid Q_i\).  The sharper first-moment estimate \eqref{eq:sharpFirstMoment} gives
\[
 \sum_{q_i\le Y}M_i^{(1)}
 \le s\sum_{\substack{I\in\cS_Y\\N(I)\ge M}}\frac1{N(I)}.
\]
Choose \(M\) large enough that this is at most \(1/4\).  Together with \eqref{eq:largePrimeQuarter} we have \(\eta_*\le1/2\).  Hence \cref{thm:abstractDistortion} yields
\[
 d(R)\ge\frac12\exp\left(-4\sum_\lambda\frac{\nu(I_\lambda)}{N(I_\lambda)}\right).
\]
Finally, \(\nu(\mathfrak p^r)=1\) for \(N\mathfrak p\le Y\) and \(\nu(\mathfrak p^r)=1+a(N\mathfrak p)\) for \(N\mathfrak p>Y\), so \(\nu(I)\le\mu_a(I)\) for every \(I\).  This proves \eqref{eq:mainDensity}.
\end{proof}

\begin{corollary}\label{cor:simpleCriterion}
If \(a(x)\to0\) and \eqref{eq:simpleA}--\eqref{eq:simpleB} hold, then the conclusion of \cref{thm:mainWeighted} holds.
\end{corollary}
\begin{proof}
Condition \eqref{eq:simpleA} makes \(A_a(x)\) bounded, so \eqref{eq:generalCriterion} follows from \eqref{eq:simpleB}.
\end{proof}

\begin{proof}[Proof of \cref{cor:explicitWeight}]
Put
\[
 a_\varepsilon(x)=\frac{(\log x)^3\ell_2(x)^{1+\varepsilon}}x.
\]
It is enough, by \cref{cor:simpleCriterion}, to prove \eqref{eq:simpleA} and \eqref{eq:simpleB}.  Split prime ideals into blocks \(e^t<N\mathfrak p\le e^{t+1}\), with \(t\) large.  By \eqref{eq:primeIdealUpper}, the number of prime ideals in the \(t\)-th block is \(O_D(e^{t+1}/t)\).  The contribution of that block to
\[
 \sum_{\mathfrak p}\frac{a_\varepsilon(N\mathfrak p)}{N\mathfrak p}
\]
is therefore
\[
 O_D\!\left(e^{-t}t^2(\log t)^{1+\varepsilon}\right),
\]
which is summable.  On the other hand, \eqref{eq:simpleB} becomes
\[
 \sum_{\mathfrak p}\frac1{N\mathfrak p\,\ell_2(N\mathfrak p)^{1+\varepsilon}},
\]
whose \(t\)-th block is
\[
 O_D\!\left(\frac1{t(\log t)^{1+\varepsilon}}\right).
\]
This is summable as well.  The finitely many small prime ideals are harmless.  Finally,
\[
 (\log\log x)^{1+\varepsilon}=o((\log x)^\varepsilon),
\]
which proves the asymptotic comparison with the weight in \cite{BBMST2022}.
\end{proof}

\section{Consequences of Axiom A}\label{sec:axiomA}
The weighted density theorem was deliberately formulated under the prime-ideal estimates actually used in the distortion argument.  We now impose the stronger structural hypothesis of Axiom~A and derive several consequences.  This produces a second, more specialised layer of the theory: Axiom~A verifies the critical analytic input automatically, gives a general weighted mean-value theorem, and reveals a subcritical regime in which no distortion is needed.

\subsection{The critical exponent \texorpdfstring{\(\rho=1\)}{rho=1}}

\begin{corollary}[Axiom A implies the critical analytic hypotheses]\label{cor:AxiomACritical}
Suppose that the ideal semigroup of \(D\) satisfies Axiom~A with exponent \(1\).  Then \eqref{eq:MertensUpper} and \eqref{eq:primeIdealUpper} hold.  Consequently, \cref{thm:mainWeighted} applies to every weight satisfying \eqref{eq:generalCriterion}, and \cref{cor:explicitWeight} applies in particular.
\end{corollary}
\begin{proof}
This is immediate from \cref{prop:AxiomAanalytic}.
\end{proof}

Thus the linear ideal-counting formula \eqref{eq:linearIdealCountIntro}, which appeared in the introduction, should be viewed as a concrete instance of Axiom~A rather than as an unrelated additional hypothesis.  In particular, for rings of integers of number fields the classical ideal-counting theorem places the weighted density result in this critical regime.

\subsection{Weighted mean values for a general Axiom A exponent}
Let \(D\) satisfy Axiom~A with exponent \(\rho>0\):
\begin{equation}\label{eq:AxiomASection}
 N_D(x)=\kappa_Dx^\rho+O_D(x^\eta),\qquad 0\le\eta<\rho.
\end{equation}
Let \(a_{\mathfrak p}\ge0\) satisfy
\begin{equation}\label{eq:apRhoSummable}
 \sum_{\mathfrak p}\frac{a_{\mathfrak p}}{(N\mathfrak p)^\rho}<\infty,
\end{equation}
and define
\[
 \mu_a(\mathfrak p^r)=1+a_{\mathfrak p},\qquad r\ge1.
\]
Set
\begin{equation}\label{eq:Harho}
 H_{a,\rho}=\prod_{\mathfrak p}\left(1+\frac{a_{\mathfrak p}}{(N\mathfrak p)^\rho}\right).
\end{equation}
The product converges by \eqref{eq:apRhoSummable}.

\begin{theorem}[Weighted mean value under Axiom A]\label{thm:meanValueRho}
Under \eqref{eq:AxiomASection} and \eqref{eq:apRhoSummable},
\begin{equation}\label{eq:meanValueRho}
 \sum_{N\mathfrak a\le x}\mu_a(\mathfrak a)
 =\kappa_DH_{a,\rho}x^\rho+o(x^\rho).
\end{equation}
Consequently, for every fixed \(K>1\),
\begin{equation}\label{eq:harmonicRho}
 \sum_{n<N\mathfrak a\le Kn}
 \frac{\mu_a(\mathfrak a)}{(N\mathfrak a)^\rho}
 =\rho\kappa_DH_{a,\rho}\log K+o(1).
\end{equation}
\end{theorem}

\begin{proof}
Define the multiplicative function \(h\) by
\[
 h(D)=1,\qquad h(\mathfrak p)=a_{\mathfrak p},\qquad
 h(\mathfrak p^r)=0\quad(r\ge2).
\]
Then
\begin{equation}\label{eq:convolutionRho}
 \mu_a=\mathbf1*h,
\end{equation}
and
\begin{equation}\label{eq:HaSeriesRho}
 \sum_{\mathfrak d}\frac{h(\mathfrak d)}{(N\mathfrak d)^\rho}
 =H_{a,\rho}<\infty.
\end{equation}
After enlarging the implied constant in \eqref{eq:AxiomASection}, if necessary, we may use that estimate for every real argument \(y\ge1\); this only changes the constant on a bounded interval.  Hence
\begin{align*}
 \sum_{N\mathfrak a\le x}\mu_a(\mathfrak a)
 &=\sum_{N\mathfrak d\le x}h(\mathfrak d)
 N_D\!\left(\frac{x}{N\mathfrak d}\right)\\
 &=\kappa_Dx^\rho
 \sum_{N\mathfrak d\le x}\frac{h(\mathfrak d)}{(N\mathfrak d)^\rho}
 +O_D\!\left(
 x^\eta\sum_{N\mathfrak d\le x}
 \frac{h(\mathfrak d)}{(N\mathfrak d)^\eta}
 \right).
\end{align*}
The first sum tends to \(H_{a,\rho}\).  Dividing the error term by \(x^\rho\) gives
\[
 \sum_{N\mathfrak d\le x}
 \frac{h(\mathfrak d)}{(N\mathfrak d)^\rho}
 \left(\frac{N\mathfrak d}{x}\right)^{\rho-\eta}.
\]
After extending the summand by zero for \(N\mathfrak d>x\), it tends pointwise to zero and is dominated by the summable series in \eqref{eq:HaSeriesRho}.  Dominated convergence therefore proves \eqref{eq:meanValueRho}.

Let
\[
 S_a(x)=\sum_{N\mathfrak a\le x}\mu_a(\mathfrak a)
 =\kappa_DH_{a,\rho}x^\rho+o(x^\rho).
\]
Partial summation with \(t^{-\rho}\) yields
\begin{align*}
 \sum_{n<N\mathfrak a\le Kn}
 \frac{\mu_a(\mathfrak a)}{(N\mathfrak a)^\rho}
 &=\frac{S_a(Kn)}{(Kn)^\rho}-\frac{S_a(n)}{n^\rho}
 +\rho\int_n^{Kn}\frac{S_a(t)}{t^{\rho+1}}\,dt\\
 &=\rho\kappa_DH_{a,\rho}\log K+o(1),
\end{align*}
which is \eqref{eq:harmonicRho}.
\end{proof}

Taking \(\rho=1\) in \cref{thm:meanValueRho} proves \cref{thm:meanValueIntro}.  Thus the mean-value theorem used in the interval application is the critical specialisation of a general Axiom~A statement.

\subsection{Interval density in the critical case}
Assume now Axiom~A with exponent \(1\) and let \(a_{\mathfrak p}=a(N\mathfrak p)\).  In addition to the hypotheses of \cref{thm:mainWeighted}, assume \eqref{eq:simpleA}; then
\[
 H_a=\prod_{\mathfrak p}\left(1+\frac{a(N\mathfrak p)}{N\mathfrak p}\right)<\infty.
\]

\begin{theorem}[Interval density for ideal moduli]\label{thm:interval}
Assume Axiom~A with exponent \(1\), the hypotheses of \cref{thm:mainWeighted}, and \eqref{eq:simpleA}.  Fix \(K>1\) and \(s\ge1\).  Then there exist \(M=M(D,a,K,s)\) and \(c=c(D,a,K,s)>0\) such that if \(\mult(\cA)\le s\) and
\[
 n<N(I_\lambda)\le Kn,\qquad n\ge M,
\]
for every \(\lambda\), then \(d(R)\ge c\).  For \(s=1\), the ideal moduli are pairwise distinct, but equal norms are allowed.
\end{theorem}
\begin{proof}
By \cref{thm:meanValueRho} with \(\rho=1\),
\[
 \sum_{n<N\mathfrak a\le Kn}\frac{\mu_a(\mathfrak a)}{N\mathfrak a}
 =\kappa_DH_a\log K+o(1).
\]
Since each ideal occurs at most \(s\) times,
\[
 \sum_\lambda\frac{\mu_a(I_\lambda)}{N(I_\lambda)}
 \le s\sum_{n<N\mathfrak a\le Kn}\frac{\mu_a(\mathfrak a)}{N\mathfrak a}
 \le s(\kappa_DH_a\log K+1)
\]
for all sufficiently large \(n\).  Increase \(M\) so that \cref{thm:mainWeighted} also applies.  One may then take
\[
 c=\frac12\exp\{-4s(\kappa_DH_a\log K+1)\}.
\]
\end{proof}

\begin{remark}
Kim \cite{Kim2009} proved number-field analogues of earlier Erd\H{o}s--Graham type density results.  The point of \cref{thm:interval} here is not a priority claim in the number-field case, but rather that the weighted argument works directly with ideals and therefore never replaces distinct ideals by distinct norms.
\end{remark}

\subsection{The subcritical range \texorpdfstring{\(0<\rho<1\)}{0<rho<1}}
The change in behaviour at \(\rho=1\) can already be seen from \cref{prop:AxiomAanalytic}.  When \(0<\rho<1\), the prime reciprocal series converges and the first moment alone controls the covering density.

\begin{theorem}[Subcritical Axiom A regime]\label{thm:subcritical}
Assume Axiom~A with exponent \(0<\rho<1\).  Fix \(s\ge1\).  For every \(\varepsilon>0\) there exists \(M=M(D,s,\varepsilon)\) such that every finite family \(\cA=\{a_\lambda+I_\lambda\}\) with \(\mult(\cA)\le s\) and
\[
 N(I_\lambda)\ge M\qquad(\lambda\in\Lambda)
\]
satisfies
\begin{equation}\label{eq:subcriticalDensity}
 d(R)\ge1-\varepsilon.
\end{equation}
\end{theorem}

\begin{proof}
By \eqref{eq:subcriticalPrimeSum},
\[
 \sum_{\mathfrak p}\frac1{N\mathfrak p-1}<\infty,
\]
so
\[
 C_D':=\prod_{\mathfrak p}\left(1+\frac1{N\mathfrak p-1}\right)<\infty.
\]
Choose \(Y\) so large that
\begin{equation}\label{eq:subcriticalHigh}
 sC_D'\sum_{N\mathfrak p>Y}\frac1{N\mathfrak p-1}<\frac{\varepsilon}{2}.
\end{equation}
Take \(\theta_i=0\) for every stage.  Then the measures are not distorted: after uniform extension, \(P_i=P_0\) for all \(i\).

For the stages with \(q_i=N\mathfrak p_i>Y\), \cref{prop:firstSecond} gives
\[
 M_i^{(1)}\le\frac{sC_D'}{q_i-1}.
\]
Hence \eqref{eq:subcriticalHigh} implies
\[
 \sum_{q_i>Y}M_i^{(1)}<\frac{\varepsilon}{2}.
\]

Let \(\cS_Y\) be the set of ideals all of whose prime factors have norm at most \(Y\).  Since only finitely many prime ideals occur,
\[
 \sum_{I\in\cS_Y}\frac1{N(I)}
 =\prod_{N\mathfrak p\le Y}\left(1-\frac1{N\mathfrak p}\right)^{-1}<\infty.
\]
For \(q_i\le Y\), the sharper estimate \eqref{eq:sharpFirstMoment} gives
\[
 \sum_{q_i\le Y}M_i^{(1)}
 \le s\sum_{\substack{I\in\cS_Y\\N(I)\ge M}}\frac1{N(I)}.
\]
Choose \(M\) so large that this is less than \(\varepsilon/2\).  Therefore
\[
 \sum_iP_0(B_i)=\sum_iM_i^{(1)}<\varepsilon.
\]
The union bound now gives
\[
 1-d(R)=P_0\!\left(\bigcup_iB_i\right)\le\sum_iP_0(B_i)<\varepsilon,
\]
proving \eqref{eq:subcriticalDensity}.
\end{proof}

\begin{remark}[Criticality of the exponent \(1\)]\label{rem:criticality}
Under Axiom~A, the abstract prime number theorem gives a qualitative transition at \(\rho=1\).  For \(0<\rho<1\), \(\sum_{\mathfrak p}(N\mathfrak p)^{-1}\) converges and \cref{thm:subcritical} shows that no distortion is necessary.  At \(\rho=1\), the reciprocal prime sum grows like \(\log\log x\), producing the logarithmic Euler-product growth that drives \cref{thm:mainWeighted}.  We do not pursue the regime \(\rho>1\) here; the present choice of distortion parameters is not designed for the resulting power-growth prime sums.  The relevance of the subcritical case to natural classes of finite-norm Dedekind domains is a separate structural question, and \cref{thm:subcritical} should be read as a conditional consequence of Axiom~A in that range.
\end{remark}

\section{Higher moments and the logarithmic scale}\label{sec:higher}
We close with a limitation of the direct integer-moment hierarchy.  This section concerns the method, not the sharp threshold of the covering problem.

Assume in this section that Axiom~A holds with exponent \(1\), so that by \cref{prop:AxiomAanalytic}
\begin{equation}\label{eq:MertensAsymp}
 \sum_{N\mathfrak p\le x}\frac1{N\mathfrak p}=\log\log x+O_D(1),
 \qquad
 \pi_D(x)\ll_D\frac{x}{\log x}.
\end{equation}
Also assume \(\sum_{\mathfrak p}a(N\mathfrak p)/(N\mathfrak p)<\infty\).  From \cref{lem:generalMoment}, the local factor for the \(k\)-th moment at \(\mathfrak p_j\) involves
\[
 \sum_{r\ge1}\frac{(r+1)^k-r^k}{q_j^r}
 =\frac{2^k-1}{q_j}+O_k(q_j^{-2}).
\]
The same Euler-product argument as before gives
\begin{equation}\label{eq:kMomentGrowth}
 E_{i-1}[\alpha_i^k]
 \ll_{D,k,s}\frac{(\log q_i)^{2^k-1}}{q_i^k}.
\end{equation}

For \(u,d>0\),
\begin{equation}\label{eq:kDistortionIneq}
 (u-d)_+\le C_k\frac{u^k}{d^{k-1}},
 \qquad C_k=\frac{(k-1)^{k-1}}{k^k}.
\end{equation}
Indeed, maximise \((t-1)/t^k\) over \(t>1\).  Hence the \(k\)-th moment bound in the distortion step carries a loss of order \(\theta_i^{-(k-1)}\).

If
\[
 a(q)\asymp\frac{(\log q)^\beta}{q},
\]
then combining \eqref{eq:kMomentGrowth} and \eqref{eq:kDistortionIneq} leads to a prime contribution of size
\[
 \frac{(\log q)^\gamma}{q},
 \qquad
 \gamma=2^k-1-\beta(k-1).
\]
If \(\gamma<0\), split the primes into blocks \(e^m<N\mathfrak p\le e^{m+1}\).  The prime-counting bound in \eqref{eq:MertensAsymp} shows that the \(m\)-th block contributes \(O_D(m^{\gamma-1})\), and the resulting series converges.  If \(\gamma\ge0\), then the weight \((\log N\mathfrak p)^\gamma\) is eventually bounded below by a positive constant, whereas the first formula in \eqref{eq:MertensAsymp} shows that \(\sum_{\mathfrak p}(N\mathfrak p)^{-1}\) diverges.  Thus the resulting prime sum is summable precisely when \(\gamma<0\), i.e.
\[
 \beta>\frac{2^k-1}{k-1}.
\]
The right side equals \(3\) for \(k=2\), \(7/2\) for \(k=3\), \(5\) for \(k=4\), and is strictly increasing for \(k\ge2\).  We therefore record the following.

\begin{proposition}\label{prop:higherMomentBarrier}
Within the direct integer-moment scheme consisting of \cref{lem:generalMoment} and the inequality \eqref{eq:kDistortionIneq}, the second moment gives the smallest logarithmic exponent.  A direct replacement by a higher integer moment does not lower the exponent \(3\).
\end{proposition}

\begin{remark}
\Cref{prop:higherMomentBarrier} is not a sharpness theorem for covering systems.  It leaves open improvements based on a different distorted measure, a sharper analysis of correlations among fibres, or another probabilistic method.
\end{remark}

\begin{problem}
Determine how close a sufficient weight can be brought to
\[
 1+\frac{(\log N\mathfrak p)^3}{N\mathfrak p},
\]
and whether the auxiliary slowly growing factor in \cref{cor:explicitWeight} can be removed.
\end{problem}

\section*{Acknowledgements and use of tools}
The distortion framework used here is adapted from \cite{BBMST2022}.  The author is responsible for verifying all mathematical statements, citations and final wording before submission.

\end{document}